\documentclass{article}
\usepackage{amsmath,amsthm,amsfonts,amssymb}
\usepackage{enumerate,multicol,empheq,hyperref,xcolor}
\usepackage[numbers,sort&compress]{natbib}

\usepackage[utf8]{inputenc} 
\usepackage[T1]{fontenc}    
\usepackage{hyperref}       
\usepackage{url}            
\usepackage{booktabs}       
\usepackage{amsfonts}       
\usepackage{nicefrac}       
\usepackage{microtype}      
\usepackage{lipsum}
\usepackage{graphicx}
\graphicspath{ {./images/} }

\theoremstyle{definition} 
\newtheorem{thm}{Theorem} 
\newtheorem{ex}{Example} 
\newtheorem{defn}{Definition} 

\newtheorem{remark}[thm]{Remark}
\newtheorem{lem}[thm]{Lemma}

\title{Efficient Computation of Laplace Residual Power Series with Explicit Coefficient Formulas}

\author{
 Pisamai Kittipoom \\
  Department of Mathematics, Faculty of Science\\
  Prince of Songkla University\\
  Hat Yai Songkhla Thailand, 90110 \\
  \texttt{pisamai.k@psu.ac.th} \\
}

\begin{document}
\maketitle
\begin{abstract}
The Residual Power Series Method (RPSM) provides a powerful framework for solving fractional differential equations. However, a significant computational bottleneck arises from the necessity of calculating the fractional derivatives of the residual function within the coefficient determination process. The Laplace Residual Power Series Method (LRPSM) partially addresses this by employing the Laplace transform. However, it introduces additional complexities and requires the computation of the residual error function at each iteration.
This work presents a novel approach that directly derives explicit formulas for the coefficients, Bypassing the need for iterative residual error function calculations. This advancement significantly enhances the computational efficiency of the method compared to both RPS method and LRPS method. \end{abstract}


\section{Introduction}\label{sec1}
Fractional differential equations have emerged as powerful tools for modeling complex physical phenomena. Due to the inherent nonlinearities of many real-world problems, exact solutions for these fractional equations are often elusive. Consequently, researchers have devoted significant effort to developing both analytical and numerical methods to approximate these solutions.
 Established techniques include homotophy perturbation \cite{homo}, cubic B-spline collocation method \cite{bsp}, Adomian decomposition method \cite{adom} and the residual power series method \cite{RPSM2014, RPSM2015,RPSM2016,RPSM2018}.

The Residual Power Series Method (RPSM) was first introduced by Abu Arqub \cite{RPSM2014} for finding approximate solutions to fractional differential equations. Unlike traditional methods, RPS method directly constructs approximate solutions as fractional power series, avoiding linearization, perturbation, and discretization techniques.
By using the residual error concept, The RPS method shows the effective for computing the coefficients of the power series. The effectiveness of the RPS method is supported by various research publications  \cite{RPSM2014,RPSM2015,RPSM2016,RPSM2018}. Shortly thereafter, Eriqat et al. \cite{lpsm2020New} combined the Laplace transform method with the RPS method and called the Laplace Residue Power Series Method (LRPSM). This method presents the solution of the equation in Laplace domain and uses the RPS method with the solution in Laplace domain. Unlike the RPS method, the LRPS method utilizes the limit in the Laplace domain of the residual error with a power term to compute the coefficients for the series solution. The LRPS approach offers a significant advantage over the RPS method by requiring simpler calculations. 

However, the key method for calculating coefficients in the LRPS method suffers from two drawbacks. Firstly, it can be cumbersome due to the repetitive need for residual error calculations at each step of the solution process. Secondly, 
a key limitation of the LRPS Method lies in its reliance on the infinite limit to determine series solution coefficients. This approach presents significant computational challenges and raises concerns about the guaranteed existence of such a limit. In contrast, our proposed method bypasses this limitation by deriving an explicit formula for the coefficients. This not only eliminates the computational burden associated with infinite limits but also offers a more robust and theoretically sound framework for solving fractional differential equations using the LRPS method.

This present article is structured as follows. Section 2 establishes the necessary background by providing key definitions and properties of fractional derivatives and Laplace transforms. Section 3 then introduces the LRPS method framework, while Section 4 highlights the key advantage of our proposed method: the straightforward calculation of coefficients. Finally, Section 5 demonstrates the efficacy of LRPS method using three illustrative examples.

\section{Preliminaries}\label{sec2}

\begin{defn}
The Caputo's time-fractional derivative of order $\alpha$ of $\psi(x,t)$ is defined as
\[ D_t^{\alpha}\psi(x,t) = I_t^{n-\alpha}\partial_t^{n} \psi(x,t), \quad n-1 < \alpha \le n, x \in I, t > \tau > 0,\]
where $I_t^{\beta}$ is the time-fractional Riemann-Liouville integral operator of order $\beta$ defined as
\[
 I_t^{\beta} \psi(x,t) = \begin{cases} \frac{1}{\Gamma(\beta)} \int_0^t (t-\tau)^{\beta-1} \psi(x, \tau)\,d\tau, &\quad \beta > 0, t > \tau \ge 0 \\ \psi(x,t),&\quad \beta = 0 
\end{cases}. 
\]
\end{defn}
In the following lemmas we review some properties of fractional derivative used throughout this paper. For more detail, we refer the reader to Ref. \cite{book1, book2, book3, book4}].
\begin{lem} \label{lem-fd}
For $n-1< \alpha \le n, \gamma > -1$ and $t \ge 0$, we have
\begin{enumerate}[(1)]
\item $D_t^{\alpha} c = 0, \quad c \in \mathbb{R}$.
\item $D_t^{\alpha} t^{\gamma} = \frac{\Gamma(\gamma+1)}{\Gamma(\gamma+1-\alpha)}\,t^{\gamma-\alpha}$.
\item $D_t^{m\alpha} t^{k\alpha} = \begin{cases} 0,&\quad k<m \\ \Gamma(k\alpha+1),&\quad k=m \\ \frac{\Gamma(k\alpha+1)}{\Gamma((k-m)\alpha+1)} t^{(k-m)\alpha},& \quad k > m \end{cases},$\qquad
where\quad $D_t^{m\alpha}=\underbrace{D_t^{\alpha}D_t^{\alpha}\dots D_t^{\alpha}}_{m\; \text{times}}.$
\end{enumerate}
\end{lem}
\begin{proof}
The proof of part (1) and (2) are in Refs \cite{book1, book2, book3, book4}
\\
To prove part (3), if $k=m$, it follows from part (2) of the lemma that   $D_t^{k\alpha}t^{k\alpha} = \Gamma(k\alpha+1)$. 
For fixed $k \in \mathbb{N}$ and $k<m$, if $m=1$ then $D_t^{\alpha} t^0 = 0$. For $m=2$, $D_t^{2\alpha}t^{\alpha} = D_t^\alpha(D^\alpha t^{\alpha}) = D_t^{\alpha} (\Gamma(\alpha+1)) = 0$. By mathematical induction, assume that $D_t^{n\alpha}t^{k\alpha} = 0$ for $m=n$. Then
\[ D_t^{(n+1)\alpha} t^{k\alpha} = D_t\left( D_t^{n\alpha}t^{k\alpha} \right) = 0. \]
Hence, $D_t^{m\alpha} t^{k\alpha} = 0$ for all $k<m$. Finally, for the case $k>m$, we refer to part (2) of the lemma. 
\end{proof}
\begin{defn}
Let $\psi(x,t)$ be a piecewise continuous function on $I \times [0,\infty)$ and of exponential order $\delta$. Then the Laplace transform of $\psi(x,t)$ in time domain is defined as 
\[ \Psi(x,s) =\mathcal{L}_t \{ \psi(x,t) \} =  \int_0^{\infty} \psi(x,t) e^{-st}\,dt,\quad s > \delta,\]
while the inverse Laplace transform of the function $\Psi(x,s)$ is defined as
\[ \psi(x,t)  =\mathcal{L}^{-1}_s \{ \Psi(x,s) \} =  \int_{Re(s)-i\infty}^{Re(s)+i\infty} \Psi(x,s) e^{st}\,ds,\quad Re(s) > \delta_0,  \]
where the Laplace transform converges absolutely at the point $\delta_0$ lying in the the right half plane.
\end{defn}
\begin{lem} (\cite{lrpsm2021soliton})\label{lem-prelim} 
 Let $\psi(x,t)$ be a piecewise continuous function on $I \times [0, \infty)$, $|\psi(x,t)| \le Me^{\delta t}$ for some $\delta, M>0$ and $\Psi(x,s)=\mathcal{L}\{ \psi(x,t) \}$. Then
\begin{enumerate}[(1)]
\item $\lim\limits_{s \to \infty} s \Psi(x,s) = \psi(x,0), \quad x \in I$
\item $\mathcal{L}\{ D^{\alpha}_t \psi(x,t)\} = s^\alpha \Psi(x,s) - \sum\limits_{k=0}^{m-1} s^{\alpha - k - 1} D^{k}_t \psi(x,t) \, \Big|_{t=0}$ \quad $m-1 < \alpha < m$.
\item $\mathcal{L}_t \{ D_t^{n\alpha}\psi(x,t) \} = s^{n\alpha} \Psi(x,s) - \sum\limits_{k=0}^{n-1}s^{(n-k)\alpha-1}D_t^{k\alpha}\psi(x,t) \, \Big|_{t=0},\quad 0< \alpha <1$ \\
where $D_t^{n\alpha}=\underbrace{D_t^{\alpha}D_t^{\alpha}\dots D_t^{\alpha}}_{n\; \text{times}}.$
\end{enumerate}
\end{lem}

\section{Constructing the LRPS solution for time-fractional differential equations} \label{sec3}
Consider the following class of time-fractional differential equation.
\begin{align} \label{high}
D_t^{m\alpha} \psi(x,t) &= \mathcal{R}[\psi(x,t)],  
\end{align}
subject to the initial conditions
\begin{equation} \label{ics}
\begin{aligned} \psi(x,0) = f_0(x), D_t^{\alpha} \psi(x,0) = f_1(x), \dots, D_t^{(m-1)\alpha} \psi(x,0) = f_{m-1}(x),   \end{aligned}
\end{equation}
where $0< \alpha \le 1$ and $\mathcal{R}[\psi(x,t)]$ is a right had side (RHS) of the equation.  

In many studies, common forms of $\mathcal{R}[\psi(x,t)]$ include integer power functions or their derivatives with respect to the spatial variable $x$. Consequently, we will assume $\mathcal{R}[\psi(x,t)]$ represents an operator expressed as:  
\begin{align} \label{nonli}
\mathcal{R}[\psi(x,t)] &= \sum_{m=1}^M\sum_{n=0}^{N} c_{m}(x) \psi^{m}(x,t) D_x^{n} \psi(x,t) + h(x,t), 
\end{align} 
where $c_n(x)$ is a coefficient function of variable $x$ and $h(x,t)$ is a nonhomogeneous function. 
The RHS term $\mathcal{R}[\psi(x,t)]$ is illustrated with examples in Table  \ref{tabu}.

\begin{table}[ht!]
\centering
\begin{tabular}{l|l}
\hline
 \qquad \qquad time-fractional equations:  & \qquad \qquad $D_t^{k\alpha } u  = \mathcal{R}[u] $ \\ \hline  
 Newell-Whitehead-Segel equation \cite{LRPSM2021}  &  $D^{\alpha} u(x,t) = u_{xx}+2u-3u^2$ 
 \\ \hline
 Burger's equation\cite{LRPSM2021} & $D^{\alpha} u(x,t) = u_{xx}-uu_x$ \\ 
 \hline
 Biological population diffusion equation \cite{LRPSM2024}  & $ D^{\alpha} u(x,y,t) = (u^2)_{xx}+(u^2)_{yy}+u(1-ru)$
   \\
  \hline
  Korteweg-de-Varies KdV equation \cite{homo2016} & $D^{\alpha} u(x,t) = - \alpha u u_{x} - \beta u_{xxx}$ 
   \\
  \hline
  Fitzhugh-Nagumo equation \cite{LRPSM2023Nagamu} & $D^{\alpha} \psi(x,t) =  [\psi]_{xx} + (1+\beta) \psi^2 + \beta \psi+ \psi^3$
  \\
  \hline
 Kolmogorov equation \cite{LRPSM2022Kolmo}& $D^{\alpha} \psi(x,t) =  (x+1)\psi_x + x^2e^t \psi_{xx}$
  \\
  \hline 
Klein-Gordon equation \cite{lrpsm2021soliton}  & $ D^{2\alpha} \psi(x,t) = \nu [\psi^2]_{xx} - \omega  [\psi^2]_{xxxx}$
\\
\hline
Burgers equation with proportional delay \cite{homdelay} & 
$D_t^{\alpha}u(x,t) = u_{xx}(x,t) + u_x\left( x, \frac{t}{2}\right)u\left( \frac{x}{2}, \frac{t}{2}\right) + \frac{1}{2}u(x,t)$
\\
\hline
 Boussinesq-Burger equations \cite{LRPSM2022} 
 & $\begin{aligned} D_t^{\alpha}u(x,t) &=-w_x - uu_x \\ D_t^{\alpha}w(x,t) &=-[uw]_x-u_{xxx}  \end{aligned}$
\\
\hline
\end{tabular}
\caption{Example of Time-Fractional DEs \label{tabu}}
\end{table}
Assume that $\psi(x,t)$ in Eq.(\ref{high}) exhibits an exponential dependence on $t$ of order $\delta$. By operating Laplace transform with respect to $t$ on Eq.(\ref{high}), we can write it in the following form
\begin{align} \label{lapHig}
\Psi(x,s) &=  \sum_{n=0}^{m-1} \frac{f_n(x)}{s^{1+n\alpha}}+\frac{1}{s^{m\alpha}} \mathcal{L} \{  \mathcal{R} [\mathcal{L}^{-1}\{ \Psi(x,s)\}]\},\quad s> \delta.
\end{align}
The function $\Psi(x,s)$ in Eq. (\ref{lapHig}) can be considered as an exact solution of (\ref{high}) in the Laplace space.

The LRPS method construct a series solution to (\ref{high}) as a fractional power series of the variable $t$ about the initial point $t=0$:
\begin{equation}\label{phi-1}
 \psi(x,t) = \sum_{n=0}^\infty \frac{\phi_n(x)}{\Gamma(n\alpha + 1)} t^{n\alpha}
\end{equation}
where $0< \alpha \le 1$ and $t > 0$.
\\
By applying the Laplace transform with respect to $t$ to the fractional power series in Eq. (\ref{phi-1}), we obtain the approximate series solution $\Psi(x,s)$ in the Laplace space as follow
\begin{align}\label{phi-2}
\Psi(x,s) 
&= \sum_{n=0}^\infty \frac{\phi_n(x)}{s^{n\alpha+1}},\quad s > \delta. 
\end{align} 

In the following lemma, we presents the coefficient  $\phi_n(x)$  in the series solutions (\ref{phi-1}) or (\ref{phi-2}).

\begin{lem} (\cite{lrpsm2021soliton}) \label{lem-coef}
Let $\psi(x,t)$ be a continuous function in $I \times [0, \infty)$ and $|\psi(x,t)| \le Me^{\delta t}$ for some $\delta, M>0$. If $\Psi(x,s) = \mathcal{L} \{ \psi(x,t)\}$ and
\[
\Psi(x,s) = \sum_{n=0}^\infty \frac{\phi_n(x)}{s^{n\alpha +1}},\quad 0< \alpha \le 1, x \in I, s > \delta, 
\]
then \quad $\phi_n(x) = D^{n\alpha}_t \psi(x,t)\,\Big|_{t=0}.$
\end{lem}
Lemma \ref{lem-coef} allows us to determine the first $m-1$ coefficients, related to the initial conditions (\ref{ics}): 
\[  
\phi_0(x) = \psi(x,0)=f_0(x),\; \phi_1(x) = D_t^{\alpha}\psi(x,0) = f_1(x), \dots, \phi_{m-1}(x) = D_t^{(m-1)\alpha}\psi(x,0) = f_{m-1}(x). 
\] 
Then Eq. (\ref{phi-2}) becomes
\begin{equation}\label{phi-2a}
\Psi(x,s) = \sum_{n=0}^{m-1} \frac{f_n(x)}{s^{1+n\alpha}}+\sum_{n=m}^{\infty} \frac{\phi_n(x)}{s^{1+n\alpha}}.
\end{equation}
For remaining coefficients $\phi_i(x), i=m, m+1, \dots$ of the series solution (\ref{phi-2a}), the LRPS method provides a framework to derive these coefficients. 

Let $\Psi_k(x,s)$, denoted as the $k-$approximate series solution, represent the $k-$th truncated series of $\Psi(x,s)$.
\begin{equation} \label{k-approx}
\Psi_k(x,s) = \sum_{n=0}^{m-1} \frac{f_n(x)}{s^{1+n\alpha}}+\sum_{n=m}^k \frac{\phi_n(x)}{s^{1+n\alpha}}.
\end{equation} 
Similar to the RPS method, we define the Laplace residual function $LRes(x,s)$ by the \textit{error} of the approximate series solution (\ref{phi-2}) and the exact solution (\ref{lapHig}) as follow:
\begin{equation} \label{resi-1}
  LRes(x,s) = \Psi(x,s) - \sum_{n=0}^{m-1} \frac{f_n(x)}{s^{1+n\alpha}} - \frac{1}{s^{m\alpha}} \mathcal{L} \{  \mathcal{R} [\mathcal{L}^{-1}\{ \Psi(x,s)\}]\}.
\end{equation}
The residual function for the $k-$th iteration or the $k-$th residual function $LRes_k(x,s)$ can be written as
\begin{align}
LRes_k(x,s) &= \Psi_k(x,s) - \sum_{n=0}^{m-1} \frac{f_n(x)}{s^{1+n\alpha}} - \frac{1}{s^{m\alpha}} \mathcal{L} \{  \mathcal{R} [\mathcal{L}^{-1}\{ \Psi_k(x,s)\}]\}. \nonumber
\\
&= \sum_{n=m}^k \frac{\phi_n(x)}{s^{1+n\alpha}}- \frac{1}{s^{2\alpha}} \mathcal{L} \left\{  \mathcal{R} \left[\mathcal{L}^{-1} \left\{ \sum_{n=0}^{m-1} \frac{f_n(x)}{s^{1+n\alpha}}+\sum_{n=m}^k \frac{\phi_n(x)}{s^{1+n\alpha}} \right\} \right] \right\},\quad k \in \mathbb{N}, \label{lr-high}
\end{align}
where $\Psi_k(x,s)$ is obtained from the truncated series (\ref{k-approx}). 
\\
To determine the coefficients $\phi_n(x)$, we use the $k-$th residual function (\ref{lr-high}) and results from the following theorem.
\begin{thm} (\cite{lrpsm2021soliton}) \label{thm-res}
Let $\psi(x,t)$ be a continuous function on $I \times [0, \infty)$ and $|\psi(x,t)| \le Me^{\delta t}$ for some $\delta, M>0$. If $\Psi(x,s) = \mathcal{L} \{ \psi(x,t)\}$ is represented in (\ref{phi-2}) and there exists a function $K$ such that $\left| s\mathcal{L}\left\{ D_t^{(n+1)\alpha} \psi(x,t)\right\} \right| \le K(x) $ on $I \times (\delta, \nu]$ for $0< \alpha \le 1$ then the error term $R_n(x,s)$ of (\ref{phi-2}) is bounded above:
\begin{equation} \label{reminder}
|R_n(x,s)| \le \frac{K(x)}{s^{(n+1)\alpha+1}},\quad x \in I, \delta< s \le \nu
\end{equation}
\end{thm}
\begin{remark}  \label{remark}
It follows by (\ref{reminder}) that for each $n \in \mathbb{N}$
\[
0 \le s^{n\alpha +1} |R_n(x,s)| \le \frac{K(x)}{s^{\alpha}}. 
\]
Therefore
\[
 \lim_{s \to \infty} s^{n\alpha +1} R_n(x,s) = 0. 
\]
It is equivalent to
\begin{equation} \label{rem}
 \lim_{s \to \infty} s^{k\alpha +1} LRes_k(x,s) = 0,\quad k = 1,2, \dots . 
\end{equation} 
We remark that this limit will play a crucial role in determining the coefficients $\phi_k(x)$.
\end{remark}

\section{Computing LRPS method Coefficients: Formulas and Procedures} 

This section discusses a difficulty in applying the LRPS method to solve time-fractional differential equations Eq. (\ref{high}). Although this method employs the $k-$th residual functions in Eq. (\ref{lr-high}) and coefficients obtained from limit analysis in Eq. (\ref{rem}), computing these coefficients often presents practical difficulties due to their mathematical complexity. Simplifying these calculations remains an under-explored area in existing research.

This study offers a significant contribution by streamlining the coefficient calculation process for the LRPS method applied to nonlinear time-fractional differential equations.  We achieve this by deriving explicit coefficient formulas.  This significantly enhances the efficiency and applicability of the LRPS method for constructing solutions to this class of equations.

 \subsection{LRPS method for the time-fractional differential equation of order (0, 1]} \label{subsec4-1}
Consider the time-fractional differential equation (\ref{high}) when $m=1$:
\begin{align} \label{lr-1}
D_t^{\alpha} \psi(x,t) &= \mathcal{R}[\psi(x,t)],\quad 0< \alpha \le 1,
\end{align}
subject to the initial condition $\psi(x,0)=f_0(x)$ and $\mathcal{R}[\psi(x,t)]$ represents an operator expressed in (\ref{nonli})
\\
We present a methodological framework for determining the LRPS coefficients, as detailed in the following steps.
\\
\textbf{Step 1} \; Find $\phi_1(x)$. 
By placing $m=1$ and $k=1$ into the residual function (\ref{lr-high}), we get
\begin{align*}
 LRes_1(x,s) 
 &= \frac{1}{s^{1+\alpha}}\phi_1(x) - \frac{1}{s^{\alpha}}\mathcal{L}_{t} \left\{ \mathcal{R} \left[\mathcal{L}_t^{-1} \left\{\frac{1}{s} f_0(x)+\frac{1}{s^{1+\alpha}} \phi_1(x) \right\} \right] \right\} \\
&=  \frac{1}{s^{1+\alpha}}\phi_1(x) - \frac{1}{s^{\alpha}}\mathcal{L}_{t} \left\{ \mathcal{R} \left[ f_0(x) +\frac{t^{\alpha}}{\Gamma(\alpha+1)} \phi_1(x)  \right] \right\}. 
\end{align*} 
By multiplying both sides of the above equation by $s^{1+\alpha}$ and applying limit as $s$ tends to infinity, we obtain
\begin{align*}
\lim_{s \to \infty} s^{1+\alpha} LRes_1(x,s) &=  \phi_1(x) - \lim_{s \to \infty} s \mathcal{L}_{t} \left\{ \mathcal{R} \left[ f_0(x)+\frac{t^{\alpha}}{\Gamma(\alpha+1)} \phi_1(x)  \right] \right\}   
\end{align*}
It follows from Eq. (\ref{rem}) that the left hand limit equals zero and we obtain
\[ 
\phi_1(x)  =  \lim_{s \to \infty} s \mathcal{L}_{t} \left\{ \mathcal{R} \left[ f_0(x)+\frac{t^{\alpha}}{\Gamma(\alpha+1)} \phi_1(x)  \right] \right\}  \]
Using part (1) of Lemma \ref{lem-prelim}, we obtain the reduced coefficient:
\begin{equation}\label{phi}
\phi_1(x) =  \mathcal{R} \left[ f_0(x)+\frac{t^{\alpha}}{\Gamma(\alpha+1)} \phi_1(x)  \right] \Bigg|_{t=0} = \mathcal{R}[f_0(x)]
\end{equation}
\allowdisplaybreaks
\\
\textbf{Step 2} \;  Find $\phi_2(x)$. By substituting $m=1$ and $k=2$ into the residual function (\ref{lr-high}), we obtain
\begin{align*}
 LRes_2(x,s)
  &= \frac{1}{s^{1+\alpha}}\phi_1(x) +\frac{1}{s^{1+2\alpha}} \phi_2(x)  \\
  &\qquad - 
  \frac{1}{s^{\alpha}}\mathcal{L}_{t} \left\{ \mathcal{R} \left[\mathcal{L}_t^{-1} 
  \left\{ \frac{1}{s} f_0(x)+\frac{1}{s^{1+\alpha}} \phi_1(x) +\frac{1}{s^{1+2\alpha}} \phi_2(x) \right\} \right] \right\} \\
  &= \frac{1}{s^{1+\alpha}}\phi_1(x) +\frac{1}{s^{1+2\alpha}} \phi_2(x)  \\
  &\qquad - 
  \frac{1}{s^{\alpha}}\mathcal{L}_{t} \left\{ \mathcal{R} \left[ f_0(x)+\frac{t^{\alpha}}{\Gamma(1+\alpha)} \phi_1(x) 
  +\frac{t^{2\alpha}}{\Gamma(1+2\alpha)} \phi_2(x)  \right] \right\}.   
\end{align*}
Multiply the above equation by $s^{1+2\alpha}$ and take limit as $s$ tends to infinity,
\begin{align} \label{phii2-0}
\lim_{s \to \infty} s^{1+2\alpha} LRes_2(x,s) &= \phi_2(x) + \lim_{s \to \infty} \Big( s^{\alpha} \phi_1(x) - s^{\alpha+1} \mathcal{L}_{t} \Big\{ \mathcal{R} \Big[ f_0(x)  +\frac{t^{\alpha}}{\Gamma(1+\alpha)} \phi_1(x) 
  +\frac{t^{2\alpha}}{\Gamma(1+2\alpha)} \phi_2(x)  \Big] \Big\} \Big).
\end{align}
Traditionally, the LRPS approach was used the limit in  Eq. (\ref{phii2-0}) to solve for the coefficient  $\phi_2(x)$. However, this method involved complex calculations and there was uncertainty about whether this limit exists or not (there is a term in $\mathcal{L}_t\{\mathcal{R}\}$ that could eliminate an unwanted term $s^{\alpha}\phi_1(x)$). This uncertainty is the main reason why we aim to simplify this limit. 
\\
In the same manner as previous step, by utilizing Eq. (\ref{rem}), we obtain
\begin{align}
\phi_2(x) &= \lim_{s \to \infty} \left(  s^{\alpha+1} \mathcal{L}_{t} \left\{ \mathcal{R} \left[ f_0(x) +\frac{t^{\alpha}}{\Gamma(1+\alpha)} \phi_1(x) 
  +\frac{t^{2\alpha}}{\Gamma(1+2\alpha)} \phi_2(x)  \right]\right\} - s^{\alpha}\phi_1(x)  \right) \nonumber \\
  &= \lim_{s \to \infty} s\left(  s^{\alpha} \mathcal{L}_{t} \left\{ \mathcal{R} \left[ f_0(x) +\frac{t^{\alpha}}{\Gamma(1+\alpha)} \phi_1(x) 
  +\frac{t^{2\alpha}}{\Gamma(1+2\alpha)} \phi_2(x)  \right]\right\} - s^{\alpha-1}\phi_1(x)  \right).  \label{phii2-1}
\end{align}
Now we will simplify this limit by observing the term $\mathcal{R}$ in Eq. (\ref{phii2-1}) at the initial point $t=0$. Based on the choice of $\mathcal{R}$ in Eq. (\ref{nonli}), we get
\begin{align*}
\mathcal{R} \left[ f_0(x)+\frac{t^{\alpha}}{\Gamma(1+\alpha)} \phi_1(x) 
  +\frac{t^{2\alpha}}{\Gamma(1+2\alpha)} \phi_2(x)  \right] \Bigg|_{t=0} &= \mathcal{R}  [f_0(x)] = \phi_1(x) .
\end{align*}  
According to part (2) in Lemma \ref{lem-prelim}, Eq. (\ref{phii2-1}) can be written as
\begin{align*}
\phi_2(x) &= \lim_{s \to \infty} s \mathcal{L}_{t} \left\{ D_t^{\alpha} \left( \mathcal{R} \left[ f_0(x)+\frac{t^{\alpha}}{\Gamma(1+\alpha)} \phi_1(x) 
  +\frac{t^{2\alpha}}{\Gamma(1+2\alpha)} \phi_2(x)  \right]\right)   \right\} .
\end{align*}
Now, to simplify the coefficient $\phi_2$ in the above equation, we use part (1) in Lemma \ref{lem-prelim}:
\begin{align} \label{explain-1}
\phi_2(x) &= D_t^{\alpha} \left( \mathcal{R} \left[ f_0(x)+\frac{t^{\alpha}}{\Gamma(1+\alpha)} \phi_1(x) 
  +\frac{t^{2\alpha}}{\Gamma(1+2\alpha)} \phi_2(x)  \right]\right) \Bigg|_{t=0}.
\end{align}
However, the coefficient $\phi_2$ in Eq. (\ref{explain-1}) can be further simplified based on the choice of $\mathcal{R}$ in Eq. (\ref{nonli}) and by utilizing part (3) of Lemma \ref{lem-fd}, we get
\[ D_t^{\alpha} \left( \mathcal{R}\left[ f(x)+\frac{t^{\alpha}}{\Gamma(1+\alpha)}\phi_1(x) + \frac{t^{2\alpha}}{\Gamma(1+2\alpha)}\phi_2(x) \right] \right) \Bigg|_{t=0} = D_t^{\alpha} \left( \mathcal{R}\left[ f(x)+\frac{t^{\alpha}}{\Gamma(1+\alpha)} \phi_1(x) \right] \right) \Bigg|_{t=0}. \]
Therefore, Eq. (\ref{explain-1}) can be reduced to  
\begin{align} \label{phii2-new}
\phi_2(x)   &= D_t^{\alpha} \left( \mathcal{R} \left[ f_0(x) +\frac{t^{\alpha} \phi_1(x)}{\Gamma(1+\alpha)}   \right]\right) \Bigg|_{t=0}.
\end{align}
In earlier works on LRPS method, finding the coefficients involved a repetitive process of calculating the limit at infinity of the residual function at each step. In the following step, we will derive a general formula that can be applied to all coefficients at once, eliminating the need for repeated calculations.
\\
\textbf{Step 3}\; Using mathematical induction, we will prove that
\begin{align} \label{formu}
\phi_k(x)
&= D_t^{(k-1)\alpha} \left( \mathcal{R} \left[ f_0(x) +\sum_{m=1}^{k-1} \frac{t^{m\alpha}\phi_m(x)}{\Gamma(1+m\alpha)} \right] \right) \Bigg|_{t=0}, \quad \text{for}\; k = 2,3,\dots 
\end{align}
is an explicit formula for determining the coefficients of the series solution (\ref{phi-2}).
\begin{proof}
Base case established in Step 2.  \\
To proceed with the inductive step, assume the following formulas hold true for the first $k-1$ coefficients:
\begin{equation} \label{induc}
 \phi_n(x)
= D_t^{(n-1)\alpha} \left( \mathcal{R} \left[ f_0(x) +\sum_{m=1}^{n-1} \frac{t^{m\alpha}\phi_m(x)}{\Gamma(1+m\alpha)} \right] \right) \Bigg|_{t=0},\qquad n=2,3,\dots,k.
\end{equation}
With the assumption that $\mathcal{R}$ is defined as in Eq. (\ref{nonli}), the previous coefficient can be evaluated at the initial time $t=0$ as
\begin{align}
 \phi_n(x)
= D_t^{(n-1)\alpha} \left( \mathcal{R} \left[ f_0(x) +\sum_{m=1}^{n+1} \frac{t^{m\alpha}\phi_m(x)}{\Gamma(1+m\alpha)} \right] \right) \Bigg|_{t=0},\qquad n=2,3,\dots,k. \label{induc-2}
\end{align}
Next we need to prove that the formula for the coefficient $\phi_{k+1}$ is true. 
\\
Similar to previous step, by calculating the $(k+1)-$th residual function and compute its limit, we obtain that  
\begin{align}
\phi_{k+1}(x)  &= \lim_{s \to \infty} s\left(  s^{k\alpha} \mathcal{L}_{t} \left\{ \mathcal{R} \left[ f_0(x) + \sum_{m=1}^{k+1}\frac{t^{\alpha m}}{\Gamma(1+m\alpha)} \phi_m(x) \right]\right\}  -  \sum_{m=1}^{k-1} s^{(k-m)\alpha -1}\phi_{m+1}(x)    \right)   \label{induc-3}
\end{align}
By placing coefficients $\phi_i, i=2,3,\dots,k$ in Eq. (\ref{induc-2}) into Eq. (\ref{induc-3}) and applying part (3) in Lemma \ref{lem-prelim}, we obtain
\begin{align*}
\phi_{k+1}(x) &= \lim_{s \to \infty} s \mathcal{L}_{t} \left\{ D_t^{k\alpha} \mathcal{R} \left[ f_0(x) +\sum_{m=1}^{k+1} \frac{t^{m\alpha}\phi_m(x)}{\Gamma(1+m\alpha)} \right] \right\}.
\end{align*}
Based on part (1) in Lemma \ref{lem-prelim}, we have
\begin{align*}
\phi_{k+1}(x) = D_t^{k\alpha} \left( \mathcal{R} \left[ f_0(x)+\sum_{m=1}^{k+1} \frac{t^{m\alpha}\phi_m(x)}{\Gamma(1+m\alpha)} \right] \right) \Bigg|_{t=0}.
\end{align*}
Thus, the proof is complete.
\end{proof}
\textbf{Step 4}\; As the final step, the series solution of Eq.(\ref{lr-1}) in The Laplace domain can be concluded by
\begin{align} \label{eq-sol-s}
\Psi(x,s) &= \frac{f_0(x)}{s} + \sum_{n=1}^\infty \frac{\phi_n(x)}{s^{\alpha n + 1}},
\end{align} 
where $\phi_n(x)$ is given by (\ref{formu}).
\\
Finally, by applying the inverse Laplace transform to Eq. (\ref{eq-sol-s}). This solution in time domain is expressed in the following form
\begin{align}
\psi(x,t) &= f_0(x) + \sum_{n=1}^\infty \frac{\phi_{n}(x)}{\Gamma(1+n\alpha)} \, t^{n\alpha}.
\end{align}
\begin{remark}
Assuming the term $\mathcal{R}$ on the RHS of Eq. (~\ref{lr-1}) is linear, specifically expressible as $\mathcal{R}[u]=L[u]$ where $L$ denotes a linear operator, the coefficient formula in Eq. (\ref{formu}) simplifies to
\begin{align*}
\phi_k(x)
&= D_t^{(k-1)\alpha} \left(  L \left[ f_0(x) + \sum_{n=1}^{k-1}  \frac{\phi_{n}(x)\,t^{n\alpha}}{\Gamma(1+n\alpha)}  \right]  \right) \Bigg|_{t=0}, \quad \text{for}\; k = 2, 3, \dots. 
    \end{align*}
Utilizing part (3) in Lemma \ref{lem-prelim}, we obtain the coefficients formula for a time-fractional linear DEs as follow
\begin{align}\label{lin-co}
    \phi_k(x) &= L[\phi_{k-1}(x)]  \Big|_{t=0},\quad k=2, 3, \dots.
\end{align}
\end{remark}
\subsection{LRPS method for the generalized time-fractional differential equations}
Consider the generalized time-fractional differential equations of the following form:
\begin{align} \label{highh}
D^{m\alpha} \psi(x,t) &= \mathcal{R}[\psi(x,t)], \quad 0< \alpha \le 1, m\in \mathbb{Z}^+,
\end{align}
subject to the initial conditions (\ref{ics}) and $\mathcal{R}[\psi(x,t)]$ denotes an operator defined in equation (\ref{nonli}).

Based on the initial conditions (\ref{ics}), we obtain the first $m-1$ coefficients for the series solution (\ref{phi-2a}). Therefore, we aim to find the subsequent coefficients $\phi_i(x),\;i=m, m+1, \dots$
\\
\textbf{Step 1}\;  Find $\phi_m(x)$, by using (\ref{lr-high}) when $k=m$, we obtain
 \begin{align*}
 LRes_m(x,s) &= \frac{\phi_m(x)}{s^{1+m\alpha}}-\frac{1}{s^{m\alpha}} \mathcal{L} \left\{  \mathcal{R} \left[ \mathcal{L}^{-1} \left\{ \sum_{n=0}^{m-1} \frac{f_n(x)}{s^{1+n\alpha}}+  \frac{\phi_m(x)}{s^{1+m\alpha}} \right\} \right] \right\} \\
 &= \frac{\phi_m(x)}{s^{1+m\alpha}}-\frac{1}{s^{m\alpha}} \mathcal{L} \left\{  \mathcal{R} \left[\sum_{n=0}^{m-1} \frac{f_n(x)\,t^{n\alpha}}{\Gamma(1+n\alpha)} + \frac{\phi_m(x) \,t^{m\alpha}}{\Gamma(1+m\alpha)} \right] \right\}
 \end{align*}
After multiplying the above equation by $s^{1+m\alpha}$ and take limit as $s$ tends to infinity, yield
\begin{align*}
\phi_m(x)&= \lim_{s \to \infty} s \mathcal{L} \left\{ \mathcal{R} \left[\sum_{n=0}^{m-1} \frac{f_n(x)\,t^{n\alpha}}{\Gamma(1+n\alpha)} + \frac{\phi_m(x) \,t^{m\alpha}}{\Gamma(1+m\alpha)} \right]\right\}.
\end{align*}
By employing part (1) of Lemma \ref{lem-prelim} and based on the choice of $\mathcal{R}$ as in Eq. (\ref{nonli}), we get
\begin{align}\label{phi2-hig}
\phi_m(x) &= \mathcal{R} \left[ \sum_{n=0}^{m-1} \frac{f_n(x)\,t^{n\alpha}}{\Gamma(1+n\alpha)} + \frac{\phi_m(x) \,t^{m\alpha}}{\Gamma(1+m\alpha)}  \right] \Bigg|_{t=0} = \mathcal{R}[f_0(x)].
\end{align}
\\
\textbf{Step 2}\; Find $\phi_{m+1}(x)$, by substituting $k=m+1$  into (\ref{lr-high}). We obtain
\begin{align*}
 LRes_{m+1}(x,s) &= \frac{\phi_m(x)}{s^{1+m\alpha}} +\frac{\phi_{m+1}(x)}{s^{1+(m+1)\alpha}}-\frac{1}{s^{m\alpha}} \mathcal{L} \left\{  \mathcal{R} \left[ \mathcal{L}^{-1} \left\{ \sum_{n=0}^{m-1} \frac{f_n(x)}{s^{1+n\alpha}} + \frac{\phi_m(x)}{s^{1+m\alpha}} + \frac{\phi_{m+1}(x)}{s^{1+(m+1)\alpha}} \right\} \right] \right\} \\
 &= \frac{\phi_m(x)}{s^{1+m\alpha}} +\frac{\phi_{m+1}(x)}{s^{1+(m+1)\alpha}}-\frac{1}{s^{m\alpha}} \mathcal{L} \left\{  \mathcal{R} \left[  \sum_{n=0}^{m-1} \frac{f_n(x)\,t^{n\alpha}}{\Gamma(1+n\alpha)} + \frac{\phi_m(x)\, t^{m\alpha}}{\Gamma(1+m\alpha)} + \frac{\phi_{m+1}(x)\,t^{(m+1)\alpha}}{\Gamma(1+(m+1)\alpha)}  \right] \right\}.
 \end{align*}
Multiplying the previous equation by $s^{1+(m+1)\alpha}$ and taking limit as $s$ approaches infinity to get
\begin{align*}
\phi_{m+1}(x) &= \lim_{s \to \infty} \left[ s^{1+\alpha} \mathcal{L} \left\{ \mathcal{R}\left[ \sum_{n=0}^{m-1} \frac{f_n(x)\,t^{n\alpha}}{\Gamma(1+n\alpha)} + \frac{\phi_m(x)\, t^{m\alpha}}{\Gamma(1+m\alpha)} + \frac{\phi_{m+1}(x)\,t^{(m+1)\alpha}}{\Gamma(1+(m+1)\alpha)}  \right] \right\} - s^\alpha \phi_m(x) \right] 
\\
&= \lim_{s \to \infty} s\left[ s^{\alpha} \mathcal{L} \left\{ \mathcal{R}\left[ \sum_{n=0}^{m-1} \frac{f_n(x)\,t^{n\alpha}}{\Gamma(1+n\alpha)} + \frac{\phi_m(x)\, t^{m\alpha}}{\Gamma(1+m\alpha)} + \frac{\phi_{m+1}(x)\,t^{(m+1)\alpha}}{\Gamma(1+(m+1)\alpha)}  \right] \right\} - s^{\alpha-1} \phi_m(x) \right].
\end{align*}
Since $\phi_m(x) = \mathcal{R}[f_0(x)] = \mathcal{R}\left[ f_0(x)+\sum\limits_{n=1}^{m-1} \frac{f_n(x)\,t^{n\alpha}}{\Gamma(1+n\alpha)} + \frac{\phi_m(x)\, t^{m\alpha}}{\Gamma(1+m\alpha)} + \frac{\phi_{m+1}(x)\,t^{(m+1)\alpha}}{\Gamma(1+(m+1)\alpha)} \right] \Bigg|_{t=0}$ and based on part (3) in Lemma \ref{lem-prelim}, we obtain that
\begin{align*}
\phi_{m+1}(x)
&= \lim_{s \to \infty} s \mathcal{L} \left\{ D_t^{\alpha} \left( \mathcal{R}\left[f_0(x)+\sum\limits_{n=1}^{m-1} \frac{f_n(x)\,t^{n\alpha}}{\Gamma(1+n\alpha)} + \frac{\phi_m(x)\, t^{m\alpha}}{\Gamma(1+m\alpha)} + \frac{\phi_{m+1}(x)\,t^{(m+1)\alpha}}{\Gamma(1+(m+1)\alpha)} \right] \right) \right\} .
\end{align*}
Based on part (1) in Lemma \ref{lem-prelim}, we get
\begin{align*}
\phi_{m+1}(x) &= D_t^{\alpha} \left( \mathcal{R}\left[ f_0(x)+\sum\limits_{n=1}^{m-1} \frac{f_n(x)\,t^{n\alpha}}{\Gamma(1+n\alpha)} + \frac{\phi_m(x)\, t^{m\alpha}}{\Gamma(1+m\alpha)} + \frac{\phi_{m+1}(x)\,t^{(m+1)\alpha}}{\Gamma(1+(m+1)\alpha)}  \right] \right) \Bigg|_{t=0}
\\
&=  D_t^{\alpha} \left( \mathcal{R}\left[ f_0(x)+\frac{f_1(x)\,t^{\alpha}}{\Gamma(1+\alpha)} + \sum\limits_{n=2}^{m-1} \frac{f_n(x)\,t^{n\alpha}}{\Gamma(1+n\alpha)} + \frac{\phi_m(x)\, t^{m\alpha}}{\Gamma(1+m\alpha)} + \frac{\phi_{m+1}(x)\,t^{(m+1)\alpha}}{\Gamma(1+(m+1)\alpha)}  \right] \right) \Bigg|_{t=0}
\end{align*}
Under the assumption of $\mathcal{R}$ in Eq. (\ref{nonli}) and using part(3) in Lemma \ref{lem-fd}, the coefficient $\phi_{m+1}(x)$ can be written as
\begin{align} \label{phi_2}
\phi_{m+1}(x) &= D_t^{\alpha} \left( \mathcal{R}\left[ f_0(x) +\frac{f_1(x)\,t^\alpha}{\Gamma(1+\alpha)} \right] \right) \Bigg|_{t=0}.
\end{align}
\\
\textbf{Step 3}\; Find $\phi_k(x),\;k=m+2, m+3, \dots$. We can skip the inductive step for brevity, as it follows the same approach presented in Step 3 of Section 4.1. Then the coefficients formula for the LRPS solution is
\begin{align} \label{phi-k-2}
\phi_k(x) &= D_t^{(k-m)\alpha} \left( \mathcal{R} \left[ \sum_{n=0}^{k-m}  \frac{\phi_{n}(x)\,t^{n\alpha}}{\Gamma(1+n\alpha)} \right] \right) \Bigg|_{t=0},\quad k=m+2, m+3, \dots,
\end{align}
where $\phi_i(x)=f_i(x), i=0,1,2,\dots, m-1$ is obtained from the initial conditions (\ref{ics}).
\section{Illustrative example}
This section highlights the effectiveness and ease of use of our explicit LRPS coefficients formula. To illustrate the applicability of our results, we consider three well-established examples: the time-fractional Kolmogorov equation, the time-fractional Klein-Gordon equation, and the time-fractional generalized Burgers equation with proportional delay. In contrast to the traditional LRPS approach, our method employs the explicit coefficients formula to construct the approximate solutions for these problems.

\begin{ex} \label{ex-1}
Consider time-fractional Kolmogorov equation of order $0< \alpha \le 1$ \cite{LRPSM2022Kolmo} 
\begin{align}
D^{\alpha} \psi(x,t) =  (x+1)\psi_x + x^2e^t \psi_{xx},\qquad \psi(x,0) = x+1.
\end{align}
The RHS of this equation exhibits linearity with variable coefficients. In particular, it takes the form: $L[\psi] = (x+1)\psi_x + x^2e^t \psi_{xx}. $
\\
By using the coefficient formula (\ref{lin-co}) for linear time-fractional equation in Remark \ref{remark}, we obtain the following coefficients:
\begin{itemize}
\item $\phi_1(x) = L[\psi(x,0)] = L[x+1] = x+1$
\item $\phi_2(x) = L\left[ (x+1)  \right]  \Big|_{t=0}.$ 
\item For $k=3, 4,\dots$, it is easy to see that the remaining coefficients 
\begin{align*}
\phi_k(x) &= x+1.
\end{align*}
\end{itemize}
The coefficients obtained here agree with those found in \cite{LRPSM2022Kolmo}.
\\
Finally, substitute theses coefficients into the $k-$th approximate series solution, we obtain the final approximate solution:
\begin{align}
 \psi_k(x,t) &=  \sum_{n=1}^k \frac{(x+1) \,t^{n\alpha}}{\Gamma(1+n\alpha)}.  
\end{align}
\end{ex}
\begin{ex} \label{ex-2}
Consider the time-fractional Klien-Gordon equation \cite{lrpsm2021soliton}
\begin{align*}
D_t^{2\alpha} \psi(x,t) = \nu [\psi^2(x,t)]_{xx} - \omega [\psi^2(x,t)]_{xxxx},
\end{align*}
subject to the initial conditions
\begin{equation} \label{ex-ics}
\begin{aligned}
 \psi(x,0) &= \frac{2\lambda^2}{3\nu} \left( 1 -\cosh\left(\sqrt{\frac{\nu}{\omega}} \frac{x}{2}\right) \right) \\
 D_t^\alpha \psi(x,0) &= \frac{\lambda^3}{3\sqrt{\nu \omega}}\sinh \left(\sqrt{\frac{\nu}{\omega}} \frac{x}{2} \right),
\end{aligned}
\end{equation}
where $0 < \alpha \le 1$, $\nu, \omega > 0$ and $\lambda \in \mathbb{R}$.
\\
The nonlinear RHS term is 
\[ \mathcal{R}[\psi] = \nu [\psi^2]_{xx}- \omega [\psi^2]_{xxxx}.\]
\begin{itemize}
\item By using the formula for the coefficient (\ref{phi2-hig}) for $\phi_2$, we have
\begin{align*}
\phi_2(x)&= \mathcal{R} \left[ \psi(x,0) \right] \\
&= \mathcal{R} \left[ \frac{2\lambda^2}{3\nu} \left( 1 -\cosh\left(\sqrt{\frac{\nu}{\omega}} \frac{x}{2}\right) \right) \right]  \\
&= \nu\left( \frac{4\lambda^4}{9\nu^2} \left( 1 -\cosh\left(\sqrt{\frac{\nu}{\omega}} \frac{x}{2}\right) \right)^2\right]_{xx} - \omega \left[ \frac{4\lambda^4}{9\nu^2} \left( 1 -\cosh\left(\sqrt{\frac{\nu}{\omega}} \frac{x}{2}\right) \right)^2 \right]_{xxxx} \\
&= -\frac{\lambda^4}{6\omega}\cosh\left(\sqrt{\frac{\nu}{\omega}} \frac{x}{2}\right),
\end{align*}
which is the same as the coefficient obtained in \cite{lrpsm2021soliton}.
\item The next step is to determine the value of $\phi_3$. Based on the coefficient formula (\ref{phi_2}), we obtain 
\begin{align*}
\phi_3(x) &=D_t^{\alpha} \left( \mathcal{R} \left[ \psi(x,0)+  D_t^\alpha \psi(x,0) \frac{t^\alpha}{\Gamma(1+\alpha)} \right] \right) \\
&= D_t^{\alpha} \left( \mathcal{R} \left[ \underbrace{\frac{2\lambda^2}{3\nu} \left( 1 -\cosh\left(\sqrt{\frac{\nu}{\omega}} \frac{x}{2}\right) \right)  + 
\frac{\lambda^3}{3\sqrt{\nu \omega}}\sinh \left(\sqrt{\frac{\nu}{\omega}} \frac{x}{2} \right) \frac{t^\alpha}{\Gamma(1+\alpha)}}_{=:\Theta_1} \right] \right)\Bigg|_{t=0}
\\
&=D_t^{\alpha} \left( \nu[\Theta_1^2]_{xx}- \omega[\Theta_1^2]_{xxxx}\right)\Big|_{t=0}.
\end{align*}
Let consider the following terms:
\begin{align*}
 \Theta_1^2 &= \frac{4\lambda^4}{9\nu^2} \left(1- \cosh\left(\sqrt{\frac{\nu}{\omega}} \frac{x}{2}\right) \right)^2 + \frac{4\lambda^5}{9\nu\sqrt{\nu\omega}}\left(1- \cosh\left(\sqrt{\frac{\nu}{\omega}} \frac{x}{2}\right) \right) \sinh\left(\sqrt{\frac{\nu}{\omega}} \frac{x}{2}\right)\frac{t^\alpha}{\Gamma(1+\alpha)}  
\\
&\qquad +\frac{\lambda^6}{9\nu\omega} \sinh\left(\sqrt{\frac{\nu}{\omega}} \frac{x}{2}\right)\frac{t^{2\alpha}}{(\Gamma(1+\alpha))^2}.
\end{align*}
Applying part(3) in Lemma \ref{lem-fd}, we obtain
\begin{align*}
D_t^{\alpha}\left[ \Theta_1^2 \right]_{xx} \Big|_{t=0} &=  \frac{4\lambda^5}{9\nu\sqrt{\nu\omega}}\left[ \left(1- \cosh\left(\sqrt{\frac{\nu}{\omega}} \frac{x}{2}\right) \right) \sinh\left(\sqrt{\frac{\nu}{\omega}} \frac{x}{2}\right) \right]_{xx} 
\\
&= \frac{\lambda^5}{9\omega\sqrt{\nu\omega}}  \left(\sinh\left(\sqrt{\frac{\nu}{\omega}} \frac{x}{2}\right) - 2\sinh\left(\sqrt{\frac{\nu}{\omega}} x\right) \right),
\end{align*}
and 
\begin{align*}
D_t^{\alpha}\left[ \Theta_1^2 \right]_{xxxx} \Big|_{t=0} &=  \frac{4\lambda^5}{9\nu\sqrt{\nu\omega}}\left[ \left(1- \cosh\left(\sqrt{\frac{\nu}{\omega}} \frac{x}{2}\right) \right) \sinh\left(\sqrt{\frac{\nu}{\omega}} \frac{x}{2}\right) \right]_{xxxx} 
\\
&= \frac{\nu\lambda^5}{36\omega^2\sqrt{\nu\omega}} \left(\sinh\left(\sqrt{\frac{\nu}{\omega}} \frac{x}{2}\right) - 8\sinh\left(\sqrt{\frac{\nu}{\omega}} x\right) \right).
\end{align*}
Therefore,
\begin{align*}
\phi_3(x) &= \nu D_t^{\alpha}\left[ \Theta_1^2 \right]_{xx} \Big|_{t=0}-\omega D_t^{\alpha}\left[ \Theta_1^2 \right]_{xxxx} \Big|_{t=0} \\
&= \frac{\lambda^5}{12} \sqrt{\frac{\nu}{\omega^3}}\sinh\left(\sqrt{\frac{\nu}{\omega}} \frac{x}{2}\right),
\end{align*}
which coincides the coefficient given in \cite{lrpsm2021soliton}, but our derivation offers a simpler approach.
\item In a similar manner, to find the coefficients $\phi_k(x)$ for $k=4,5,\dots$, we utilize the coefficient formula (\ref{phi-k-2}) for $m=2$
\begin{align*}
\phi_k(x) &= D_t^{(k-2)\alpha} \left( \mathcal{R} \underbrace{\left[ \sum_{n=0}^{k-2} \phi_n(x)
 \frac{t^{n\alpha}}{\Gamma(1+n\alpha)}\right]}_{=:\Theta_{k-2}}  \right) \;\Bigg|_{t=0} 
 \\
 &=\nu \,D_t^{(k-2)\alpha} \left( [\Theta_{k-2}^2]_{xx} \right) \Big|_{t=0}
 - \omega \,D_t^{(k-2)\alpha} \left( [\Theta_{k-2}^2]_{xxxx} \right) \Big|_{t=0},
\end{align*}
where $\phi_{n}, n=0,1$ is obtained from the initial conditions (\ref{ex-ics}) and  $\phi_{n}, n=2,3,\dots,k-2$ are obtained from the previous calculations.
\\
Consider
\begin{align*}
D_t^{(k-2)\alpha}  \left( \Theta_{k-2}^2 \right) \Big|_{t=0} &= D_t^{(k-2)\alpha}\left( \left( \sum_{n=0}^{k-2} \phi_n(x)
 \frac{t^{n\alpha}}{\Gamma(1+n\alpha)}\right)^2\right)\; \Bigg|_{t=0} 
\\
&= D_t^{(k-2)\alpha}\left( \left( \sum_{m=0}^{k-2}\sum_{n=0}^{k-2} 
 \frac{t^{(n+m)\alpha}}{\Gamma(1+m\alpha)\Gamma(n\alpha)}\phi_{m}(x)\phi_n(x)\right)^2\right)\; \Bigg|_{t=0} 
 \\
&=\sum_{m=0}^{k-2}  \frac{\Gamma(1+(k-2)\alpha)}{\Gamma(1+m\alpha)\Gamma(1+(k-2-m)\alpha)} \phi_m(x)\phi_{k-2-m}(x).
\end{align*}

Therefore, we obtain the formula for the coefficients of this problem: 
\begin{align}
\phi_k(x) &= \nu \sum_{m=0}^{k-2}  \frac{\Gamma(1+(k-2)\alpha)}{\Gamma(1+m\alpha)\Gamma(1+(k-2-m)\alpha)} [\phi_m(x)\phi_{k-2-m}(x)]_{xx} \nonumber \\
 & \quad  \qquad   - \omega \sum_{m=0}^{k-2}  \frac{\Gamma(1+(k-2)\alpha)}{\Gamma(1+m\alpha)\Gamma(1+(k-2-m)\alpha)} [\phi_m(x)\phi_{k-2-m}(x)]_{xxxx},\quad k=4,5,\dots.\label{sosoli}
\end{align}
Maple software will be employed to determine the subsequent coefficients based on the formula (\ref{sosoli}) as show below
\begin{align*}
\phi_4(x) &= -\frac{\lambda^5 \nu}{24\omega^2} \cosh \left( \sqrt{\frac{\nu}{\omega}} \frac{x}{2} \right)
\\
\phi_5(x) &= \frac{\lambda^7}{48}\sqrt{\frac{\nu^3}{\omega^5}} \sinh \left( \sqrt {{\frac {v}{w}}} \frac{x}{2} \right), 
\\
\phi_6(x) &= -\frac{\lambda^8 \nu^2}{96\omega^3} \cosh \left( \sqrt {{\frac {v}{w}}} \frac{x}{2} \right),
\\
\phi_7(x) &= \frac{\lambda^9}{192}\sqrt{\frac{\nu^5}{\omega^7}} \sinh \left( \sqrt {{\frac {v}{w}}} \frac{x}{2} \right)
\end{align*}

These coefficients are in line with the derivation presented in \cite{lrpsm2021soliton} (see Application 4.1, for reference). However, our formula proposes a simplified approach that achieves the same coefficients.
The derivation of the series solution is omitted for brevity, interested readers can find it in (\cite{lrpsm2021soliton}).
\end{itemize}
\end{ex}
\begin{ex}  \label{ex-delay}
Consider the time-fractional generalized Burgers equation with proportional delay \cite{homdelay}
\begin{align} \label{delay}
D_t^{\alpha}u(x,t) &= u_{xx}(x,t) + u_x\left( x, \frac{t}{2}\right)u\left( \frac{x}{2}, \frac{t}{2}\right) + \frac{1}{2}u(x,t),
\end{align}
subject to $u(x,0) = x$, where $0< \alpha \le 1.$
\\
The RHS nonlinear-delay term is $\mathcal{R}[u] = u_{xx}(x,t) + u_x\left( x, \frac{t}{2}\right)u\left( \frac{x}{2}, \frac{t}{2}\right) + \frac{1}{2}u(x,t)$.
\\
Utilizing coefficients formula (\ref{formu}) in Section \ref{subsec4-1}, we obtain
\begin{itemize}
\item $\phi_1(x) = \mathcal{R}[u(x,0)] = \mathcal{R}[x] = x$.
\item $\phi_2(x) = D_t^{\alpha}\left( \mathcal{R} \left[ x + \frac{\phi_1(x)\,t^{\alpha}}{\Gamma(1+\alpha)} \right] \right)\Bigg|_{t=0}$, 
\begin{align*}
\phi_2(x) &= D_t^{\alpha}\left( \mathcal{R} \left[ x + \frac{x\,t^{\alpha}}{\Gamma(1+\alpha)} \right] \right)\Bigg|_{t=0} = D_t^{\alpha}\left( \mathcal{R} \left[ x \left(1 + \frac{t^{\alpha}}{\Gamma(1+\alpha)}\right) \right] \right)\Bigg|_{t=0}
\\
&= D_t^{\alpha}\left(   \frac{x}{2}  \left(1 + \frac{t^{\alpha}}{2^{\alpha}\Gamma(1+\alpha)}\right)^2 +  \frac{x}{2}  \left(1 + \frac{t^{\alpha}}{\Gamma(1+\alpha)}\right) \right)\Bigg|_{t=0}
\\
&=x \left(\frac{1}{2^{\alpha}}+\frac{1}{2} \right).
\end{align*}
For ease of subsequent calculation, we define $\phi_2(x) = c_1(\alpha)$ where $c_1(\alpha) = \left(\frac{1}{2^{\alpha}}+\frac{1}{2} \right).$
\item $\phi_3(x) = D_t^{2\alpha}\left( \mathcal{R} \left[  x + \frac{\phi_1(x)\,t^{\alpha}}{\Gamma(1+\alpha)} +
\frac{\phi_2(x)\,t^{2\alpha}}{\Gamma(1+2\alpha)}   \right] \right)\Bigg|_{t=0}$, 
\begin{align*}
\phi_3(x) 
&= D_t^{2\alpha}\left( \mathcal{R} \left[ x \left( 1 + \frac{t^{\alpha}}{\Gamma(1+\alpha)} + \left(\frac{1}{2^{\alpha}}+\frac{1}{2} \right)\frac{t^{2\alpha}}{\Gamma(1+2\alpha)} \right) \right] \right)\Bigg|_{t=0} 
\\
&= D_t^{2\alpha}\left( \frac{x}{2}\left( 1 + \frac{t^{\alpha}}{2^{\alpha}\Gamma(1+\alpha)} + \left(\frac{1}{2^{\alpha}}+\frac{1}{2} \right)\frac{t^{2\alpha}}{2^{2\alpha}\Gamma(1+2\alpha)} \right)^2 \right.
\\
&\qquad \qquad \qquad \qquad \left.
+ \frac{x}{2} \left( 1 + \frac{t^{\alpha}}{\Gamma(1+\alpha)} + \left(\frac{1}{2^{\alpha}} +\frac{1}{2} \right)\frac{t^{2\alpha}}{\Gamma(1+2\alpha)} \right) \right)\Bigg|_{t=0}
\\
&= D_t^{2\alpha}\left( \frac{x\,t^{2\alpha}}{2^{2\alpha+1}(\Gamma(1+\alpha))^{2}} + x\left(\frac{1}{2^{\alpha}}+\frac{1}{2} \right)\frac{t^{2\alpha}}{2^{2\alpha}\Gamma(1+2\alpha)} +\frac{x}{2}\left(\frac{1}{2^{\alpha}}+\frac{1}{2} \right)\frac{t^{2\alpha}}{\Gamma(1+2\alpha)}  \right)\Bigg|_{t=0}
\\
&= x \left( \frac{\Gamma(1+2\alpha)}{2^{2\alpha+1}(\Gamma(1+\alpha))^{2}} 
+ \left( \frac{1}{2^{2\alpha}}+\frac{1}{2} \right)\left(\frac{1}{2^{\alpha}}+\frac{1}{2} \right) \right).
\end{align*}
To simplify subsequent calculations, let 
\[ \phi_3(x) = x c_2(\alpha),\quad \text{where}\quad c_2(\alpha) = \left( \frac{\Gamma(1+2\alpha)}{2^{2\alpha+1}(\Gamma(1+\alpha))^{2}} 
+ \left( \frac{1}{2^{2\alpha}}+\frac{1}{2} \right)\left(\frac{1}{2^{\alpha}}+\frac{1}{2} \right) \right). \]
\end{itemize} 
The remaining coefficients can be computed in the same manner:
\begin{align*}
\phi_4(x)&= xc_3(\alpha),\;
\text{where}\;  c_3(\alpha) =\left( \frac{1}{2^{3\alpha}}+\frac{1}{2} \right)c_2(\alpha)+ 
\frac{\Gamma(1+3\alpha)}{2^{3\alpha}\Gamma(1+\alpha)\Gamma(1+2\alpha)} c_1(\alpha). 
\\
\phi_5(x) &= xc_4(\alpha),\;
\text{where}\;  c_4(\alpha) =\left( \left( \frac{1}{2^{4\alpha}}+\frac{1}{2} \right)c_3(\alpha)+ 
\frac{\Gamma(1+4\alpha)c_2(\alpha)}{2^{4\alpha}\Gamma(1+\alpha)\Gamma(1+3\alpha)}  + \frac{\Gamma(1+4\alpha)(c_1(\alpha))^2}{2^{4\alpha+1}(\Gamma(1+2\alpha))^2} \right). 
\end{align*}

Therefore, the $5-$th approximate series solution is
\begin{align} \label{apr-sol}
u_5(x,t) &= x\left( 1+ \frac{t^{\alpha}}{\Gamma(1+\alpha)} +c_1(\alpha)\frac{ t^{2\alpha}}{\Gamma(1+2\alpha)} 
+ c_2(\alpha) \frac{ t^{3\alpha}}{\Gamma(1+3\alpha)} + c_3(\alpha) \frac{ t^{4\alpha}}{\Gamma(1+4\alpha)} +  c_4(\alpha) \frac{ t^{5\alpha}}{\Gamma(1+5\alpha)}\right).
\end{align}
\begin{table}[ht!] 
\centering
\begin{tabular}{cccc}
\hline
$x$ & $t$  & HPM  & LRPS method 
\\
& & Abs.error (fourth-order) & Abs. error (fifth approx. series sol.) 
\\ \hline 
$0.25$ & $0.25$ &  $0.000002123$ & $8.88 \times 10^{-8}$ \\
$0.25$ & $0.50$ &  $0.000070943$ & $0.0000058388$ \\
$0.25$ & $0.75$ &  $0.000563483$ & $0.0000690968$ \\
$0.25$ & $1.00$ &  $0.002487123$ & $0.0004037913$ \\
$0.50$ & $0.25$ &  $0.000004245$ & $1.775 \times 10^{-7}$ \\
$0.50$ & $0.50$ &  $0.000141885$ & $0.0000116765$ \\
$0.50$ & $0.75$ &  $0.001126970$ & $0.000138194$ \\
$0.50$ & $1.00$ &  $0.004974250$ & $0.000807587$ \\
$0.75$ & $0.25$ &  $0.000006367$ & $2.662 \times 10^{-7}$  \\
$0.75$ & $0.50$ &  $0.000212830$ & $0.000017512$ \\
$0.75$ & $0.75$ &  $0.001690450$ & $0.000207295$ \\
$0.75$ & $1.00$ &  $0.007461370$ & $0.001211370$ \\
\hline
\end{tabular}
\caption{The absolute error of $4^{th}$-order homotopy perturbation solution \cite{homdelay} and $5^{th}$-approximate LRPS series solution for Example \ref{ex-delay} with $\alpha=1$ \label{tab-delay}}
\end{table}

To validate the accuracy of the LRPS method and the proposed coefficients formula for approximating solutions to Eq. (\ref{delay}), Table \ref{tab-delay} presents a comparison of absolute errors. The table compares the LRPS method with the homotopy perturbation method (HPM) \cite{homdelay} for solutions in the interval $[0, 1]$ at $\alpha=1$. As evident from Table \ref{tab-delay}, the absolute error of the fifth-approximate LRPS approximate solution (Eq. (\ref{apr-sol})) is demonstrably smaller than the errors from the HPM. 
\end{ex}
\section{Conclusion}
This paper presented the Laplace Residual Power Series Method (LRPSM) as a tool for obtaining analytical solutions to fractional differential equations. We demonstrated that LRPS coefficients can be calculated using the concept of the limit at infinity of the residual function with a power function. While both LRPS method and its predecessor, RPS method, require calculating the residual function at each step during the solution process, our work offers a key advantage.

Our method eliminates repetitive calculations of the residual function by deriving a direct formula for all coefficients at once. This formula, similar to the RPS method, involves fractional derivatives. While computing fractional derivatives are considered a disadvantage of the RPS method, our approach offers a key distinction. In our formula, the fractional derivative of order $n\alpha$ is applied to a sum of power functions of order  $k\alpha$, where $n,k \in \mathbb{N}$. This significantly simplifies the computation of these derivatives, making our method much more efficient compared to the traditional LRPS method's repetitive calculations of the residual function.

To demonstrate the efficiency of our method, we applied it to three specific examples:  the time-fractional Kolmogorov equation, the time-fractional Klein-Gordon equation and the time-fractional generalized Burgers equation with proportional delay. In Example \ref{ex-1} and \ref{ex-2}, we were able to derive explicit formulas for the coefficients, avoiding the use of fractional derivatives. This significantly reduces the computational complexity compared. 


\begin{thebibliography}{10}

\bibitem{homo}
Zaid~M. O.
\newblock Solitary solutions for the nonlinear dispersive k(m,n) equations with
  fractional time derivatives.
\newblock {\em Physics Letters A}, 370(3):295--301, 2007.

\bibitem{bsp}
W.~K. Zahra, W.~A. Ouf, and M.~S. EL-AZAB.
\newblock Cubic b-spline collocation algorithm for the numerical solution of
  newell-whitehead-segel type equations.
\newblock {\em Electronic Journal of Mathematical Analysis and Applications},
  2(82):81--100, 2014.

\bibitem{adom}
A.M. Wazwaz.
\newblock Construction of soliton solutions and periodic solutions of the
  boussinesq equation by the modified decomposition method.
\newblock {\em Chaos, Solitons and Fractals}, 12(8):1549--1556, 2001.

\bibitem{RPSM2014}
M.~Alquran.
\newblock Analytical solutions of fractional foam drainage equation by residual
  power series method.
\newblock {\em Math Sci}, 8(1):153–160, 2014.

\bibitem{RPSM2015}
L~Wang and X~Chen.
\newblock Approximate analytical solutions of time fractional
  whitham–broer–kaup equations by a residual power series method.
\newblock {\em Entropy}, 17:6519--6533, 2015.

\bibitem{RPSM2016}
F.~Tchier, M.~Inc, Z.~S. Korpinar, and D.~Baleanu.
\newblock Solutions of the time fractional reaction–diffusion equations with
  residual power series method.
\newblock {\em Advances in Mechanical Engineering}, 8(10):1--10, 2018.

\bibitem{RPSM2018}
R.~J. Mohan and S.~Chakraverty.
\newblock A new iterative method based solution for fractional black–scholes
  option pricing equations (bsope).
\newblock {\em SN Applied Sciences}, 1(55):1--10, 2019.

\bibitem{lpsm2020New}
T.~Eriqat, A.~El-Ajou, M.N. Oqielat, Z.~Al-Zhour, and S.~Momani.
\newblock A new attractive analytic approach for solutions of linear and
  nonlinear neutral fractional pantograph equations.
\newblock {\em Chaos, Solitons and Fractals}, 138(1):1--11, 2020.

\bibitem{book1}
K.~Oldham and J.~Spanier.
\newblock {\em The Fractional Calculus: Theory and Applications of
  Differentiation and Integration to Arbitrary Oders}.
\newblock Academic Press, 1974.

\bibitem{book2}
K.~Miller and B.~Ross.
\newblock {\em An Introduction to Fractional Calculus and Fractional
  Differential Equations}.
\newblock Wiley, 1993.

\bibitem{book3}
J.~Podlubny.
\newblock {\em Fractional Differential Equations}.
\newblock Academic Press, 1999.

\bibitem{book4}
A.~Kilbas, H.~Srivastava, and J.~Trujillo.
\newblock {\em Theory and Applications of Fractional Differential Equations}.
\newblock Elsevier, 2006.

\bibitem{lrpsm2021soliton}
A.~El-Ajou.
\newblock Adapting the laplace transform to create solitary solutions for the
  nonlinear time-fractional dispersive pdes via a new approach.
\newblock {\em Eur. Phys. J. Plus}, 136(229):1--15, 2021.

\bibitem{LRPSM2021}
M.~Alquran, M.~Alsukhour, M.~Ali, and I.~Jaradat.
\newblock Combination of laplace transform and residual power series techniques
  to solve autonomous n-dimensional fractional nonlinear systems.
\newblock {\em Nonlinear Engineering}, 10(1):282--292, 2021.

\bibitem{LRPSM2024}
R.~Pant, G.~Aroraand, B.~K. Singh, and H.~Emadifar.
\newblock Numerical solution of two-dimensional fractional differential
  equations using laplace transform with residual power series method.
\newblock {\em Nonlinear Engineering}, 13(1):1--15, 2024.

\bibitem{homo2016}
S.~Arshad, A.~Sohail, and Khadija Maqbool.
\newblock Nonlinear shallow water waves: A fractional order approach.
\newblock {\em Alexandria Engineering Journal}, 55:525--532, 2016.

\bibitem{LRPSM2023Nagamu}
A.~S. Alshehry.
\newblock A comparative analysis of laplace residual power series and a new
  iteration method for fitzhugh-nagumo equation in the caputo operator
  framework.
\newblock {\em Fractal and Fractional}, 7(867):1--19, 2023.

\bibitem{LRPSM2022Kolmo}
H.~Aljarrah, H.~Alaroud, A.~Ishak, and M.~Darus.
\newblock Approximate solution of nonlinear time-fractional pdes by laplace
  residual power series method.
\newblock {\em Mathematics}, 10(12):1--16, 2022.

\bibitem{homdelay}
M.~G Sakar, F.~Uludag, and F.~Erdogan.
\newblock Numerical solution of time-fractional nonlinear pdes with
  proportional delays by homotopy perturbation method.
\newblock {\em Applied Mathematical Modelling}, 40(13-14):6639--6649, 2016.

\bibitem{LRPSM2022}
A.~Sarhan, A.~Burqan, R.~Saadeh, and Z.~Al-Zhour.
\newblock Analytical solutions of the nonlinear time-fractional coupled
  boussinesq-burger equations using laplace residual power series technique.
\newblock {\em Fractal and Fractional}, 6(11):2504--3110, 2022.

\end{thebibliography}

\end{document}